\documentclass[final,11pt,a4paper]{amsart} 

\usepackage{soul}
\usepackage[dvipsnames]{xcolor}
\usepackage{graphicx}
\usepackage[all]{xy}
\usepackage{placeins}
\usepackage{enumitem}
\usepackage{amssymb}
\usepackage{latexsym}
\usepackage{amsmath}
\usepackage{mathrsfs}
\usepackage{array,booktabs}
\usepackage{verbatim}
\usepackage{fullpage}
\usepackage{textcomp}
\usepackage{color}
\usepackage{cancel}
\usepackage{amsfonts}
\usepackage{showkeys}
\usepackage{psfrag,epsfig}
\usepackage{mathtools}
\usepackage{float}

\usepackage{hyperref}

\newcommand{\hyref}[2]{ \hyperref[#2]{#1~\ref*{#2}} }

\newtheorem{thmIntro}{Theorem}    
\usepackage{stackengine,graphicx,amsmath,fp,scalerel}
\newcount\crosswd
\newcount\termwd
\newcommand\rawcrossout[2]{\ensurestackMath{%
  \setbox0=\hbox{$#2$}%
  \crosswd=\wd0\relax%
  \setbox0=\hbox{$#1$}%
  \termwd=\wd0\relax%
  \FPdiv\myscale{\the\termwd}{\the\crosswd}%
  \stackengine{0pt}{#1}{\stretchrel*{\scalebox{\myscale}[1]{#2}}{#1}}{O}{c}{F}{T}{L}}}
\def\XX{\kern-3pt/}
\def\YY{\kern-.5pt}

\newcommand{\Canakci}{\c{C}anak\c{c}\i}
\newcommand{\Ilke}{\.{I}lke }
\newcommand{\I}{\.{I}}

\numberwithin{figure}{section}
\numberwithin{table}{section}

\theoremstyle{plain}
\newtheorem{theorem}{Theorem}[section]

\newtheorem{lemma}[theorem]{Lemma}

\theoremstyle{definition}

\newtheorem{remark}[theorem]{Remark}

\definecolor{dgreen}{rgb}{0,0.45,0}
\definecolor{beaublue}{rgb}{0.74, 0.83, 0.9}
\definecolor{darklavender}{rgb}{0.45, 0.31, 0.59}
\definecolor{darkorchid}{rgb}{0.6, 0.2, 0.8}
\definecolor{darkpastelpurple}{rgb}{0.59, 0.44, 0.84}
\definecolor{electricviolet}{rgb}{0.56, 0.0, 1.0}

\DeclareMathAlphabet{\mathpzc}{OT1}{pzc}{m}{it}

\newcommand{\cala}{\mathcal{A}}
\newcommand{\calb}{\mathcal{B}}

\newcommand{\calo}{\mathcal{O}}

\newcommand{\arr}{\ar@{-}[r]}

\renewcommand{\phi}{\varphi}
\renewcommand{\epsilon}{\varepsilon}
\usepackage{tikz}
\usepackage{tikz-cd}
\usetikzlibrary{snakes}
\usetikzlibrary{shapes.geometric,positioning}
\usetikzlibrary{arrows,decorations.pathmorphing,decorations.pathreplacing}
\usetikzlibrary{matrix, calc, arrows}
\usetikzlibrary{decorations.pathmorphing,shapes}
\usetikzlibrary{decorations.pathreplacing}
\newcounter{sarrow}

\tikzset{join/.code=\tikzset{after node path={%
			\ifx\tikzchainprevious\pgfutil@empty\else(\tikzchainprevious)%
			edge[every join]#1(\tikzchaincurrent)\fi}}}

\tikzset{>=stealth',every on chain/.append style={join},
	every join/.style={->}}

\tikzset{vertex/.style={circle,fill=black,inner sep=1pt,outer sep=2pt},
	tinyvertex/.style={font=\scriptsize,minimum size=6pt},
	smallvertex/.style={inner sep=1pt, font=\small},
	>=stealth',
	leadsto/.style={-angle 90,decorate,decoration=snake,very thick},
	cut/.style={decorate,decoration=saw,very thick}}
\tikzset{
	partial ellipse/.style args={#1:#2:#3}{
		insert path={+ (#1:#3) arc (#1:#2:#3)}
	}
}
\begin{document}

\title[Bases for cluster algebras from orbifolds with one marked point]{Bases for cluster algebras from orbifolds with one marked point}
\thanks{The first author was partially supported by EPSRC grants EP/N005457/1 and EP/P016014/1. Most of the work was carried out while both authors were affiliated with Durham University.}

\author{\Ilke  \Canakci}
\address{School of Mathematics, Statistics and Physics, Newcastle University, NE1 7RU, United Kingdom.}
\email{ilke.canakci@newcastle.ac.uk}

\author{Pavel Tumarkin}
\address{Department of Mathematical Sciences, Durham University, Lower Mountjoy, Stockton Road, Durham, DH1 3LE, United Kingdom.}
\email{pavel.tumarkin@durham.ac.uk}

\begin{abstract} We generalize the construction of the bangle, band and bracelet bases for cluster algebras from orbifolds to the case where there is only one marked point on the boundary.
\end{abstract}

\maketitle

{\small
\setcounter{tocdepth}{2}
\tableofcontents
}
\section{Introduction}

Cluster algebras were introduced by Fomin and Zelevinsky~\cite{FZ1} in the context of canonical bases. A cluster algebra is a commutative ring with a distinguished set of generators ({\em cluster variables}), which are grouped into overlapping finite collections of the same cardinality ({\em clusters}) connected by local transition rules ({\em mutations}).

An important problem in cluster algebra theory is a construction of good bases.
In~\cite{MSW2} Musiker, Schiffler and Williams constructed two types of bases  ({\em bangle} basis $\calb^\circ$ and {\em bracelet} basis $\calb$) for cluster algebras originating from unpunctured surfaces~\cite{FST,FT,FG1} with principal coefficients~\cite{FZ4}. A {\em band} basis $\calb^\sigma$ was introduced by D.~Thurston in~\cite{T}. All the three bases are parametrized by collections of mutually non-intersecting arcs and closed loops, and all their elements are {\em positive}, i.e. the expansion of any basis element in any cluster is a Laurent polynomial with non-negative coefficients.

In~\cite{FeTu}, Felikson and Tumarkin generalized the constructions of all the three bases to cluster algebras originating from unpunctured orbifolds with orbifold points of weight $1/2$: these are counterparts of surface cluster algebras in the non-skew-symmetric case~\cite{FeSTu3}. However, although the bases can be easily defined for any such orbifold, the proof in~\cite{FeTu} required the orbifold to have at least two boundary marked points. 

The main goal of this note is to remove the assumption on the number of marked points. Namely, we prove the following theorem.

\begin{thmIntro}[Theorem~\ref{Thm:Main}] Let $\cala(\calo)$ be the cluster algebra with principal coefficients associated to an orbifold $\calo$ with exactly one marked point on its boundary. Then the bangles $\calb^\circ$, the bands $\calb^\sigma$, and the bracelets $\calb$ are bases of the cluster algebra $\cala(\calo)$.
\end{thmIntro}

The proof mainly follows the one of~\cite[Theorem 2]{CLS} where the authors  proved a similar statement for cluster algebras from unpunctured surfaces with exactly one marked point. In particular, the main ingredient is the following result.

\begin{thmIntro}[Theorem~\ref{Thm:LgenArb}] The Laurent polynomial associated to the essential loop around the boundary in an orbifold with one marked point belongs to the cluster algebra $\cala(\calo)$. 
\end{thmIntro}

To prove Theorem~\ref{Thm:LgenArb}, we use snake graph calculus~\cite{CS, CS2,CS3}. One of the features of our proof is that, unlike \cite{CLS}, we do not require the coefficients to be invertible.
 
The paper is organized as follows. In Section \ref{Sec:SnakeGraphs}, we introduce snake and band graphs associated to curves in a triangulated orbifold with orbifold points of weight $1/2$. Section \ref{Sec:Genus1} is devoted to show that the Laurent polynomial associated to the essential loop around the boundary lies in the cluster algebra when the genus is $1$. Section \ref{Sec:ArbitraryGenus} generalizes the result of Section \ref{Sec:Genus1} to arbitrary genus, and finally in Section \ref{Sec:ArbitraryLoop} we show that Laurent polynomials associated to all essential loops belong to the cluster algebra and hence give bases for cluster algebras associated to orbifolds with $1$ marked point on their boundary.  

{\it Acknowledgements:} The authors would like to thank the organisers of the workshop on Quivers \& Bipartite Graphs in London in $2016$ where this project was initiated, and also to Ralf Schiffler for encouraging them to work on this project.

\section{Snake graphs associated to triangulations of orbifolds}
\label{Sec:SnakeGraphs}
In this section, we generalize the snake graph formula of \cite{MSW, MSW2} to the orbifold case. The authors of \cite{MSW} associate snake graph and band graph to every arc (may have self-crossings) and essential loop in the surface, respectively, and give an explicit formula for Laurent polynomials in the cluster algebra corresponding to these curves. This expansion formula is parametrized by {\em perfect matchings} of snake and band graphs, see Section 3.1 in \cite{MSW}.

We first show that the cluster variable associated to a pending arc $\tau'$ with local configuration given in Figure~\ref{Fig:FlipPending} can be given as an instance of \cite{MSW} expansion formula. Let $\tau$ be the pending arc corresponding to the mutation (flip) of $\tau'$. We consider an unfolding as in Figure~\ref{Fig:FlipPending} (see \cite{FeSTu2}) and associate the corresponding snake graph $G_{\tau}.$    

The exchange relation corresponding to the mutation of the pending arc $\tau'$ coincides with the perfect matching formula, that is
$$x_{\tau}x_{\tau'}=x^2_a+x^2_b,$$ 
where $x^2_a$ is the associated weight of the perfect matching with two vertical edges of $G_{\tau}$ and $x^2_b$ is the associated weight of the perfect matching with two horizontal edges of $G_{\tau}$ in the \cite{MSW} formula. This expansion formula can be extended for arcs and closed loops in orbifolds without any alterations using unfoldings \cite{FeSTu2}. Namely, the procedure of assigning a snake graph to a non-pending arc is the following: we take a connected component of a lift of the arc in any unfolding, take the corresponding snake graph in the obtained surface cluster algebra, and then specialize variables to get a snake graph for the initial arc. This is always possible since an unfolding exists for any orbifold with boundary~\cite{FeSTu3}. Assigning a band graph to a closed (or semi-closed) loop requires a bit more accuracy: we need an existence of such a lift $\gamma'$ of the loop $\gamma$ in the unfolding that the Laurent polynomial of $\gamma$ is obtained from the one of $\gamma'$ by a specialization of variables. This is the case when the number of connected components of $\gamma'$ is equal to the degree of the covering giving rise to the unfolding. However, such an unfolding always exists in view of \cite[Lemma 4.7]{FeTu}, so the same procedure as for arcs holds for band graphs as well.

\begin{figure} [ht]
	\includegraphics[width=.75\textwidth]{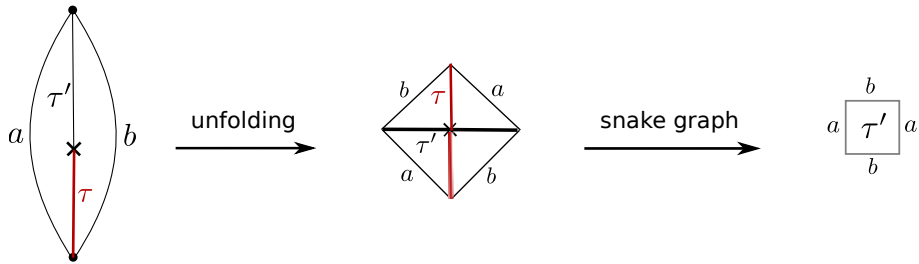}
	\psfrag{a}{$a$}
	\psfrag{b}{$b$}
	\psfrag{T}{$\tau$}
	\psfrag{T'}{$\tau'$}
	\caption{Flip of a pending arc $\tau'$, its unfolding and the corresponding snake graph $G_{\tau}$.}
	\label{Fig:FlipPending}
\end{figure}
\section{Genus $1$ case}
\label{Sec:Genus1}

The aim of this section is to prove the following theorem.  

\begin{theorem} \label{Thm:Lgen1}
Let $\calo$ be an orbifold of genus $1$ and $\cala(\calo)$ be its associated cluster algebra with principal coefficients.
The Laurent polynomial associated to the essential loop around the boundary belongs to the cluster algebra $\cala(\calo).$
\end{theorem}
\begin{proof} We choose a particular triangulation $T$ of the orbifold $\calo$ of genus $1$ with exactly one marked point on its boundary given in Figure~\ref{Fig:Gen1multOrb}. Let $L$ be the Laurent polynomial associated to the closed loop around the boundary. We want to show $L$ is in the cluster algebra $\cala(\calo).$ We will show this by applying resolutions and grafting of snake and band graphs associated to the arcs indicated in Figure~\ref{Fig:Gen1multOrb}.  We refer to \cite{CS,CS2,CS3} for the snake graph calculus tools we use.
\vspace{-3mm}
\begin{figure} [ht]
	\includegraphics[width=.21\textwidth]{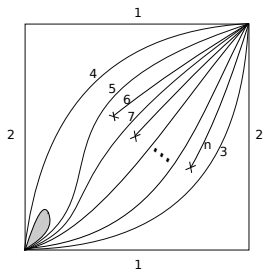}
	\psfrag{U}{$U$}
	\psfrag{V}{$V$}
	\caption{Initial triangulation for the surface with several orbifold points.}
	\psfrag{U}{$U$}
	\psfrag{V}{$V$}
	\label{Fig:Gen1multOrb}
\end{figure}

We will abuse notation and denote arcs in the orbifold and the corresponding Laurent polynomials of these arcs by the same capital letter. 

In Figure~\ref{Fig:SGLoopGen1}, we explicitly indicate each grafting or self-grafting of snake and band graphs associated to smoothings of crossings indicated in Figure~\ref{Fig:Gen1_T}. Note that, for the sake of simplicity, these identities are considered only when the surface has precisely one orbifold point. We also remark that identities given in Figure~\ref{Fig:SGLoopGen1} are valid in the associated cluster algebra $\cala(\calo)$ when considered with \emph{trivial coefficients}. We use snake graph calculus to lift these identities to cluster algebras with \emph{principal coefficients}.
\vspace{-3mm}
\begin{figure}[!b]
	\includegraphics[width=.52\textwidth]{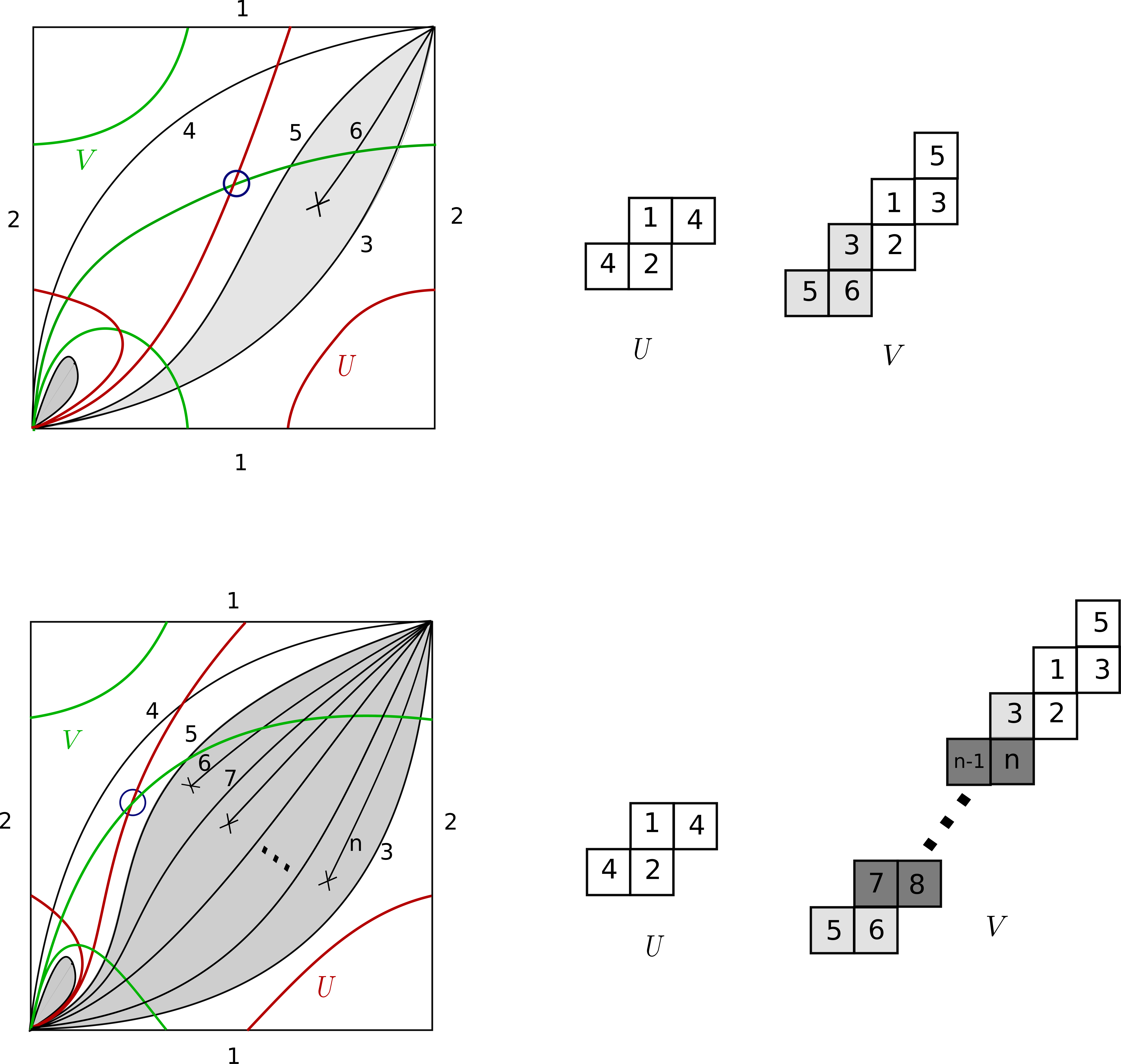}
	\caption{Snake graphs associated to the arcs $U$ and $V$ for surfaces with one orbifold point (top) and for a surface with many orbifold points (bottom).}
	\label{Fig:Arc_SG}
\end{figure}

\begin{figure}[!t] 
	\includegraphics[width=.9\textwidth]{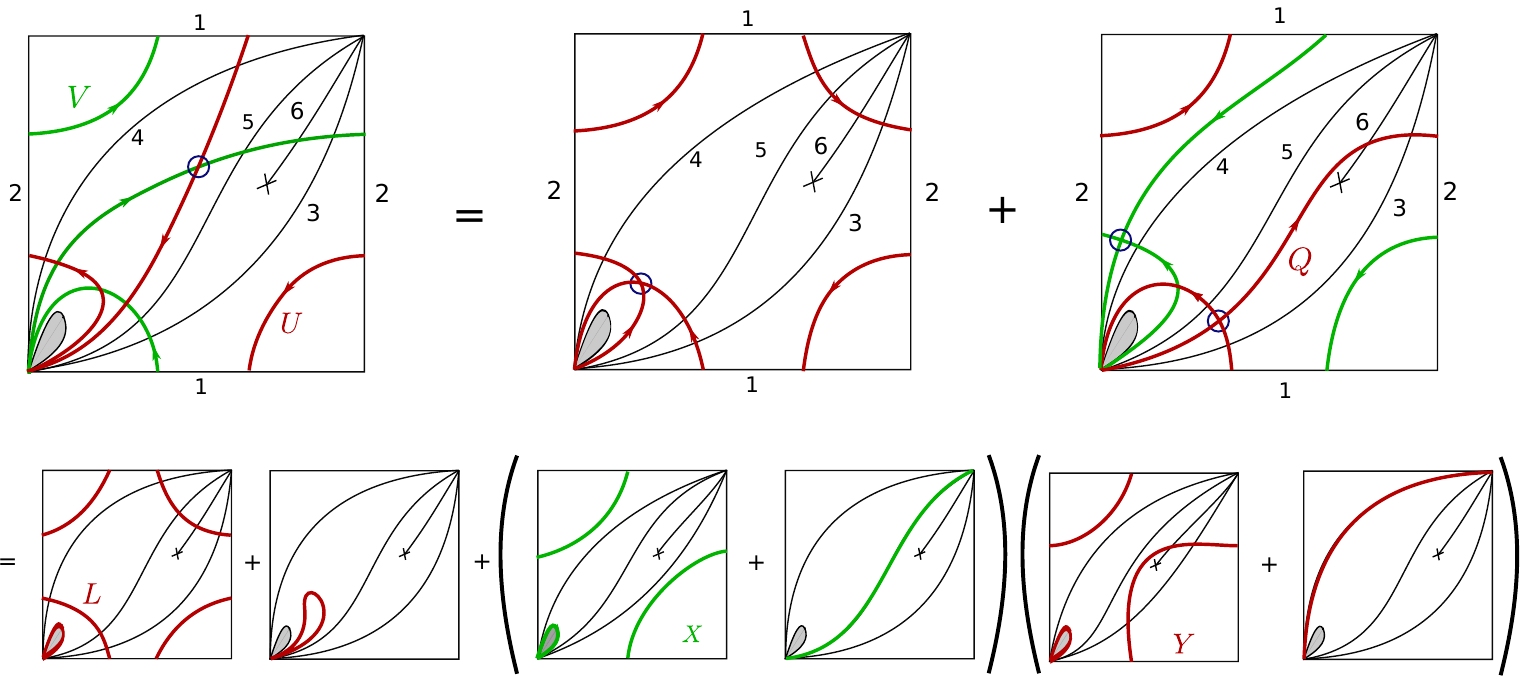}
	\caption{Geometric realization of the identities given in Figure~\ref{Fig:SGLoopGen1} in terms of arcs in the surface. Step 1: smoothing of the crossing of the arcs $V$ (green) and $U$ (red), Step 2: smoothings of the three self-crossing arcs simultaneously.}
\label{Fig:Gen1_T}
\end{figure}

Comparing snake graphs associated to the arcs $U$ and $V$ when the surface has one orbifold point and when it has several orbifold points (Figure~\ref{Fig:Arc_SG}), one can observe that the corresponding snake graphs associated to the arc $U$ is exactly the same for both surfaces, and the snake graph associated to $V$ differs by a zig-zag piece corresponding to the crossings of the arcs $7,8,\dots,n$ when the surface has several orbifold points. Therefore, the identities in terms of snake graphs will be similar when there are several orbifold points in the surface, which can be recovered by replacing every zig-zag snake subgraph $5,6,3$ by the zig-zag snake graph $5,6,7,\dots,n,3$.

More precisely, the identity in Figure~\ref{Fig:Gen1_T} translates to 
$$UV=L+(y_4X+x_5)(y_5Y+y_5\tilde y x_4)$$
where $\tilde y= \displaystyle\prod_{i\mbox{ tile in } G_{Q}} y_i$ and the curve $Q$, indicated in Figure~\ref{Fig:Gen1_T}, is obtained by smoothing $U$ and $V.$ In particular, if the genus is $1$ then $ \tilde y=y_1y_2y^2_3y_5y_6$.


\newsavebox{\myboxV}
\sbox{\myboxV}{%
   \begin{tikzpicture}[scale=.5]
    \draw (0,0)--(1,0)--(2,0)--(2,2)--(1,2)--(1,0) (0,0)--(0,1)--(3,1)--(3,2)--(2,2);
    \node at (.5,.5){$4$};
    \node at (1.5,.5){$2$};
    \node at (1.5,1.5){$1$};
    \node at (2.5,1.5){$4$};
    \node at (0,.5)[left, scale=.6]{$B$};
    \node at (.5,0)[below, scale=.6]{$5$};
    \node at (.5,1)[above, scale=.6]{$2$};

    \node at (2.5,2)[above, scale=.6]{$B$};    
    \node at (3,1.5)[right, scale=.6]{$5$};
    \node at (2.5,1)[below, scale=.6]{$1$};
    \end{tikzpicture}
 }

\newsavebox{\myboxW}
\sbox{\myboxW}{%
\begin{tikzpicture}[scale=.5]
    \draw (0,0)--(1,0)--(2,0)--(2,2)--(1,2)--(1,0) (0,0)--(0,1)--(3,1)--(3,2)--(2,2) (2,2)--(2,3)--(4,3)--(4,2)--(3,2)--(3,4)--(4,4)--(4,3);
    \node at (.5,.5){$5$};
    \node at (1.5,.5){$6$};
    \node at (1.5,1.5){$3$};
    \node at (2.5,1.5){$2$};
    \node at (2.5,2.5){$1$};
    \node at (3.5,2.5){$3$}; 
    \node at (3.5,3.5){$5$}; 
    \node at (.5,0)[below, scale=.6]{$B$};
    \end{tikzpicture}
}

\newsavebox{\myboxK}
\sbox{\myboxK}{%
\begin{tikzpicture}[scale=.5]
    \draw (0,0)--(1,0)--(2,0)--(2,2)--(1,2)--(1,0) (0,0)--(0,1)--(3,1)--(3,2)--(2,2) (2,2)--(2,3)--(4,3)--(4,2)--(3,2)--(3,4)--(4,4)--(4,3)--(5,3)--(5,5)--(4,5)--(4,4)--(6,4)--(6,6)--(5,6)--(5,5)--(6,5);
    \node at (.5,.5){$4$};
    \node at (1.5,.5){$2$};
    \node at (1.5,1.5){$1$};
    \node at (2.5,1.5){$4$};
    \node at (2.5,2.5){$5$};
    \node at (3.5,2.5){$6$}; 
    \node at (3.5,3.5){$3$}; 
    \node at (4.5,3.5){$2$}; 
    \node at (4.5,4.5){$1$}; 
    \node at (5.5,4.5){$3$}; 
    \node at (5.5,5.5){$5$}; 
    \end{tikzpicture}
}

\newsavebox{\myboxS}
\sbox{\myboxS}{%
\begin{tikzpicture}[scale=.5]
    \draw (0,0)--(1,0)--(2,0)--(2,2)--(1,2)--(1,0) (0,0)--(0,1)--(2,1);
    \node at (.5,.5){$4$};
    \node at (1.5,.5){$2$};
    \node at (1.5,1.5){$1$};
    \node at (0,.5)[left, scale=.6]{$B$};
    \node at (.5,0)[below, scale=.6]{$5$};
    \node at (.5,1)[above, scale=.6]{$2$};
    \node at (1.5,2)[above, scale=.6]{$4$};
    \end{tikzpicture}
}  

\newsavebox{\myboxT}
\sbox{\myboxT}{%
\begin{tikzpicture}[scale=.5]
    \draw (1,0)--(2,0)--(2,2)--(1,2)--(1,0) (1,1)--(3,1)--(3,2)--(2,2) (2,2)--(2,3)--(4,3)--(4,2)--(3,2)--(3,4)--(4,4)--(4,3);
    \node at (1.5,.5){$6$};
    \node at (1.5,1.5){$3$};
    \node at (2.5,1.5){$2$};
    \node at (2.5,2.5){$1$};
    \node at (3.5,2.5){$3$}; 
    \node at (3.5,3.5){$5$}; 
   \node at (4,3.5)[right, scale=.6]{$4$};
   \node at (3.5,4)[above, scale=.6]{$B$};
    \end{tikzpicture}
}

\newsavebox{\myboxL}
\sbox{\myboxL}{%
\begin{tikzpicture}[scale=.5]
    \draw (0,0)--(1,0)--(2,0)--(2,2)--(1,2)--(1,0) (0,0)--(0,1)--(3,1)--(3,2)--(2,2) (2,2)--(2,3)--(4,3)--(4,2)--(3,2)--(3,4)--(4,4)--(4,3)--(5,3)--(5,5)--(4,5)--(4,4)--(6,4)--(6,6)--(5,6)--(5,5)--(6,5);
    \node at (.5,.5){$4$};
    \node at (1.5,.5){$2$};
    \node at (1.5,1.5){$1$};
    \node at (2.5,1.5){$4$};
    \node at (2.5,2.5){$5$};
    \node at (3.5,2.5){$6$}; 
    \node at (3.5,3.5){$3$}; 
    \node at (4.5,3.5){$2$}; 
    \node at (4.5,4.5){$1$}; 
    \node at (5.5,4.5){$3$}; 
    \node at (5.5,5.5){$5$}; 
    \node at (0,0){$\bullet$};     
    \node at (0,1){$\bullet$};
    \node at (6,5){$\bullet$};     
    \node at (6,6){$\bullet$};
    \end{tikzpicture}
}

\newsavebox{\myboxX}
\sbox{\myboxX}{%
\begin{tikzpicture}[scale=.5]
    \draw (0,0)--(1,0)--(1,0)--(1,2)--(0,2)--(0,0) (0,1)--(1,1);
    \node at (.5,.5){$1$};
    \node at (.5,1.5){$2$};
    \node at (0,0){$\bullet$};     
    \node at (1,0){$\bullet$};
    \node at (0,2){$\bullet$};     
    \node at (1,2){$\bullet$};
    \end{tikzpicture}
}

\newsavebox{\myboxY}
\sbox{\myboxY}{%
\begin{tikzpicture}[scale=.5]
    \draw (1,0)--(2,0)--(2,2)--(1,2)--(1,0) (1,1)--(3,1)--(3,2)--(2,2) (2,2)--(2,3)--(4,3)--(4,2)--(3,2)--(3,3);
    \node at (1.5,.5){$6$};
    \node at (1.5,1.5){$3$};
    \node at (2.5,1.5){$2$};
    \node at (2.5,2.5){$1$};
    \node at (3.5,2.5){$3$}; 
    \node at (1,0){$\bullet$};     
    \node at (2,0){$\bullet$};
    \node at (4,3){$\bullet$};     
    \node at (4,2){$\bullet$};
    \end{tikzpicture}
}

\newsavebox{\myboxB}
\sbox{\myboxB}{%
\begin{tikzpicture}[scale=.5]
\draw (0,0)--(1,0);
\node at (.5,0)[above, scale=.6]{$B$};
\end{tikzpicture}
}

\newsavebox{\myboxFour}
\sbox{\myboxFour}{%
\begin{tikzpicture}[scale=.5]
\draw (0,0)--(1,0);
\node at (.5,0)[above, scale=.6]{$4$};
\end{tikzpicture}
}

\newsavebox{\myboxFive}
\sbox{\myboxFive}{%
\begin{tikzpicture}[scale=.5]
\draw (0,0)--(1,0);
\node at (.5,0)[above, scale=.6]{$5$};
\end{tikzpicture}
}

\newsavebox{\myboxST}
\sbox{\myboxST}{%
\begin{tikzpicture}
\matrix [matrix of math nodes,left delimiter=(,right delimiter=)]
{\begin{tikzpicture}
[scale=.5]
\node at (0,0){$ \usebox{\myboxS}$}; 
\node at (4,0){$ \usebox{\myboxT}$}; 
\end{tikzpicture}
\\
};
\end{tikzpicture}
}

\newsavebox{\myboxLB}
\sbox{\myboxLB}{%
\begin{tikzpicture}
\matrix [matrix of math nodes,left delimiter=(,right delimiter=)]
{\begin{tikzpicture}
[scale=.5]
\node at (0,0){$ \usebox{\myboxL}$}; 
\node at (4,3){$ \usebox{\myboxB}$}; 
\end{tikzpicture}
\\
};
\end{tikzpicture}
}

\newsavebox{\myboxXBF}
\sbox{\myboxXBF}{%
\begin{tikzpicture}
\matrix [matrix of math nodes,left delimiter=(,right delimiter=)]
{\begin{tikzpicture}
[scale=.5]
\node[red] at (-1,1.5){$y_4$};
\node at (.5,0){$ \usebox{\myboxX}$}; 
\node at (2,1.5){$ \usebox{\myboxB}$};
\node at (3,1.3){$+$};
\node at (4,1.5){$ \usebox{\myboxFive}$}; 
\end{tikzpicture}
\\
};
\end{tikzpicture}
}

\newsavebox{\myboxYBF}
\sbox{\myboxYBF}{%
\begin{tikzpicture}
\matrix [matrix of math nodes,left delimiter=(,right delimiter=)]
{\begin{tikzpicture}
[scale=.5]
\node[red] at (-1.2,1.5){$y_5$};
\node at (1,0){$ \usebox{\myboxY}$}; 
\node at (3,1.5){$ \usebox{\myboxB}$};
\node at (4.2,1.3){$+$};
\node[red] at (7,1.5){$y_5y_1y_2y^2_3y_5y_6$};
\node at (10,1.5){$ \usebox{\myboxFour}$}; 
\end{tikzpicture}
\\
};
\end{tikzpicture}
}

\begin{figure}[!b]
\scalebox{.82}{\begin{tikzpicture}
	\node at (0,0){$ \usebox{\myboxV}$}; 
	\node at (2.5,0){$ \usebox{\myboxW}$};
	\node at (4,0){$=$};  
	\node at (5.5,0){$ \usebox{\myboxK}$}; 
	\node at (7.5,0){$ \usebox{\myboxB}$}; 
	\node at (8.2,0){$+$}; 
	\node[red] at (9,0){$y_5$};
	\node at (12,0){$ \usebox{\myboxST}$}; 
	
	\node at (-1,-4){$=$};  
	\node[scale=.95] at (2,-4){$ \usebox{\myboxLB}$}; 
	\node at (5,-4){$+$}; 
	\node[scale=.95] at (7.5,-4){$ \usebox{\myboxXBF}$}; 
	\node[scale=.95] at (13.5,-4){$ \usebox{\myboxYBF}$}; 
	\end{tikzpicture}
}
\caption{First line: grafting at $B$; second line: grafting at $B$, self-grafting at $4$, self-grafting at $5$, respectively.}
\label{Fig:SGLoopGen1}
\end{figure}

Therefore, in order to show that the essential loop around the boundary is in the cluster algebra, namely $L\in\cala(\calo)$, it is sufficient to show $y_4X\in\cala(\calo)$ and $y_5Y\in\cala(\calo).$  These are given in Lemma~\ref{Lem:XY}.
\end{proof}

\begin{lemma} \label{Lem:XY} With the notation in Theorem~\ref{Thm:Lgen1}, the following hold.
		\begin{enumerate}
		\item $y_4X\in\cala(\calo),$
		\item $y_5Y\in\cala(\calo).$
	\end{enumerate}
\end{lemma}

\begin{proof} The proof is given by successively applying snake graph calculus to some chosen crossing arcs. For simplicity, we explicitly indicate each step in Figure~\ref{Fig:y4X} where we consider the triangulation  of the orbifold with one orbifold point. The general case is given in exactly the same steps by considering the triangulation of the orbifold as in Figure~\ref{Fig:Gen1multOrb} and the identities in terms of snake graphs carried out in a similar way as discussed in the proof of Theorem~\ref*{Thm:Lgen1}. 
	

\begin{figure} [!t]
	\includegraphics[width=1.05\textwidth]{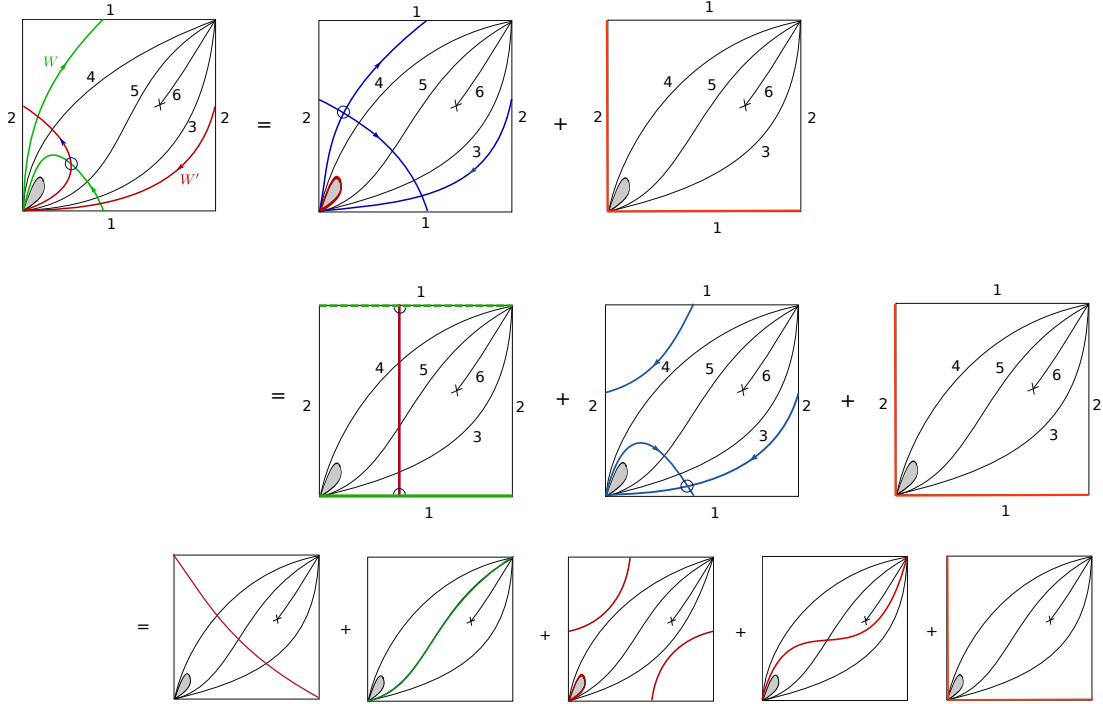}
	\psfrag{W}{$W$}
	\psfrag{W'}{$W'$}
	\caption{Geometric realization of the identities given in Figure~\ref{Fig:SGy4X} in terms of arcs in the surface. Step 1: smoothing of the crossing of the arcs $W$ (green) and $W'$ (red); step 2: smoothing of a self-crossing; step 3: grafting with $1$, smoothing of a self-crossing.}
	\label{Fig:y4X}
\end{figure}

The identities in the cluster algebra with principal coefficients corresponding to smoothings of crossings of the arcs in Figure~\ref{Fig:y4X} are indicated in Figure~\ref{Fig:SGy4X}. This proves $y_4X\in\cala(\calo)$.

\newsavebox{\myboxC}
\sbox{\myboxC}{%
	\begin{tikzpicture}[scale=.5]
	\draw (0,0)--(1,0)--(1,2)--(0,2)--(0,0) (0,1)--(2,1)--(2,2)--(1,2);
	\node at (.5,.5){$1$};
	\node at (.5,1.5){$3$};
	\node at (1.5,1.5){$5$};
	\node at (2,1.5)[right, scale=.6]{$B$};
	\end{tikzpicture}
}  

\newsavebox{\myboxD}
\sbox{\myboxD}{%
	\begin{tikzpicture}[scale=.5]
	\draw (0,0)--(2,0)--(2,1)--(0,1)--(0,0) (1,0)--(1,1);
	\node at (.5,.5){$4$};
	\node at (1.5,.5){$2$};
	\node at (0,.5)[left, scale=.6]{$B$};
	\end{tikzpicture}
}  

\newsavebox{\myboxE}
\sbox{\myboxE}{%
	\begin{tikzpicture}[scale=.5]
	\draw (0,0)--(1,0)--(1,2)--(0,2)--(0,0) (0,1)--(2,1)--(2,2)--(1,2) (2,1)--(4,1)--(4,2)--(2,2) (3,1)--(3,2);
	\node at (.5,.5){$1$};
	\node at (.5,1.5){$3$};
	\node at (1.5,1.5){$5$};
	\node at (2.5,1.5){$4$};
	\node at (3.5,1.5){$2$};
	\node at (.5,0)[below, scale=.6]{$2$};
	\node at (2.5,2)[above, scale=.6]{$2$};
	\end{tikzpicture}
} 

\newsavebox{\myboxF}
\sbox{\myboxF}{%
	\begin{tikzpicture}[scale=.5]
	\draw (0,0)--(1,0)--(1,2)--(0,2)--(0,0) (0,1)--(2,1)--(2,2)--(1,2) (2,1)--(3,1)--(3,2)--(2,2);
	\node at (.5,.5){$1$};
	\node at (.5,1.5){$3$};
	\node at (1.5,1.5){$5$};
	\node at (2.5,1.5){$4$};
	\node at (0,0){$\bullet$};
	\node at (1,0){$\bullet$};
	\node at (2,2){$\bullet$};
	\node at (3,2){$\bullet$};
	\end{tikzpicture}
}

\newsavebox{\myboxG}
\sbox{\myboxG}{%
	\begin{tikzpicture}[scale=.5]
	\draw (0,0)--(2,0)--(2,1)--(0,1)--(0,0) (1,0)--(1,2)--(3,2)--(3,1)--(2,1)--(2,2);
	\node at (.5,.5){$5$};
	\node at (1.5,.5){$3$};
	\node at (1.5,1.5){$1$};
	\node at (2.5,1.5){$2$};
	\node at (0,.5)[left, scale=.6]{$B$};
	\end{tikzpicture}
}

\newsavebox{\myboxH}
\sbox{\myboxH}{%
	\begin{tikzpicture}[scale=.5]
	\draw (0,0)--(3,0)--(3,1)--(0,1)--(0,0) (1,0)--(1,1) (2,0)--(2,1);
	\node at (.5,.5){$3$};
	\node at (1.5,.5){$5$};
	\node at (2.5,.5){$4$};
	\end{tikzpicture}
} 

\newsavebox{\myboxTF}
\sbox{\myboxTF}{%
\begin{tikzpicture}[scale=.5]
\draw (0,0)--(1,0)--(1,1)--(0,1)--(0,0);
\node at (.5,.5){$5$};
\end{tikzpicture}	
}

\newsavebox{\myboxXX}
\sbox{\myboxXX}{%
	\begin{tikzpicture}[scale=.5]
	\draw (0,0)--(1,0)--(1,0)--(1,2)--(0,2)--(0,0) (0,1)--(1,1);
	\node at (.5,.5){$1$};
	\node at (.5,1.5){$2$};
	\node at (0,0){$\bullet$};     
	\node at (1,0){$\bullet$};
	\node at (0,2){$\bullet$};     
	\node at (1,2){$\bullet$};
	\end{tikzpicture}
}

\newsavebox{\myboxBB}
\sbox{\myboxBB}{%
	\begin{tikzpicture}[scale=.5]
	\draw (0,0)--(1,0);
	\node[scale=.6] at (.5,.5){$B$};
	\end{tikzpicture}
}

\newsavebox{\myboxXBT}
\sbox{\myboxXBT}{%
	\begin{tikzpicture}
	\matrix [matrix of math nodes,left delimiter=(,right delimiter=)]
	{\begin{tikzpicture}
		[scale=.5]
		\node at (.5,0){$ \usebox{\myboxXX}$}; 
		\node at (2,1.5){$ \usebox{\myboxBB}$};
		\node at (3,1.3){$+$};
		\node[red] at (4.5,1.3){$y_1y_2y_3$};
		\node at (6.4,1){$ \usebox{\myboxTF}$}; 
		\end{tikzpicture}
		\\
	};
\end{tikzpicture}
}

\begin{figure}[h]
	\scalebox{.82}{\begin{tikzpicture}
		\node at (0,0){$ \usebox{\myboxC}$}; 
		\node at (1.7,.25){$ \usebox{\myboxD}$};
		\node at (3,0){$=$};  
		\node at (4.5,0){$ \usebox{\myboxE}$}; 
		\node at (6,0){$+$}; 
		\node[red] at (7,0){$y_1y_3y_5$};
		\node at (8,.2){$\begin{tikzpicture}[scale=.5]
			\draw (0,0)--(1,0);	
			\node at (.5,0)[above, scale=.6]{$2$};
			\end{tikzpicture}$}; 
		\node at (8.7,.2){$\begin{tikzpicture}[scale=.5]
			\draw (0,0)--(1,0);	
			\node at (.5,0)[above, scale=.6]{$1$};
			\end{tikzpicture}$}; 
		
		\node at (3,-2){$=$};  
		\node at (4.2,-2){$\usebox{\myboxF}$}; 
		\node at (5.6,-2){$\begin{tikzpicture}[scale=.5]
				\draw (0,0)--(1,0);	
				\node at (.5,0)[above, scale=.6]{$1$};
				\end{tikzpicture}$}; 
		\node at (6.5,-2){$+$}; 
		\node[red] at (7,-2){$y_4$};
		\node at (8.2,-2){$\usebox{\myboxG}$};
		\node at (9.5,-2){$+$}; 
		\node[red] at (10.5,-2){$y_1y_3y_5$};
		\node at (11.5,-1.8){$\begin{tikzpicture}[scale=.5]
			\draw (0,0)--(1,0);	
			\node at (.5,0)[above, scale=.6]{$2$};
			\end{tikzpicture}$}; 
		\node at (12.2,-1.8){$\begin{tikzpicture}[scale=.5]
			\draw (0,0)--(1,0);	
			\node at (.5,0)[above, scale=.6]{$1$};
			\end{tikzpicture}$};

		\node at (3,-4){$=$};
		\node[red] at (3.5,-4){$y_1$};
		\node at (4.6,-4){$\usebox{\myboxH}$};   
		\node at (5.6,-4){$+$};
		\node at (6.3,-4){$\begin{tikzpicture}[scale=.5]
				\draw (0,0)--(1,0);	
				\node at (.5,0)[above, scale=.6]{$5$};
				\end{tikzpicture}$};
		\node at (7,-4){$+$}; 
		\node[red] at (7.5,-4){$y_4$};
		\node at (10.4,-4){$\usebox{\myboxXBT}$};
		\node at (13.2,-4){$+$}; 
		\node[red] at (14,-4){$y_1y_3y_5$};  
		\node at (15,-4){$\begin{tikzpicture}[scale=.5]
				\draw (0,0)--(1,0);	
				\node at (.5,0)[above, scale=.6]{$2$};
				\end{tikzpicture}$}; 
		\node at (15.7,-4){$\begin{tikzpicture}[scale=.5]
				\draw (0,0)--(1,0);	
				\node at (.5,0)[above, scale=.6]{$1$};
				\end{tikzpicture}$};

	\end{tikzpicture}
	}
	\caption{First line: grafting at $B$; second line: smoothing of a self-crossing; third line: grafting with $1$ and a self-grafting, respectively.}
	\label{Fig:SGy4X}
\end{figure}

Similarly, we show $y_5Y\in\cala(\calo)$ in Figure~\ref{Fig:y5Y} when the surface has one orbifold point and the geometric realization of these identities are given in Figure~\ref{Fig:SGy5Y}.

\begin{figure} [ht]
	\includegraphics[width=1.01\textwidth]{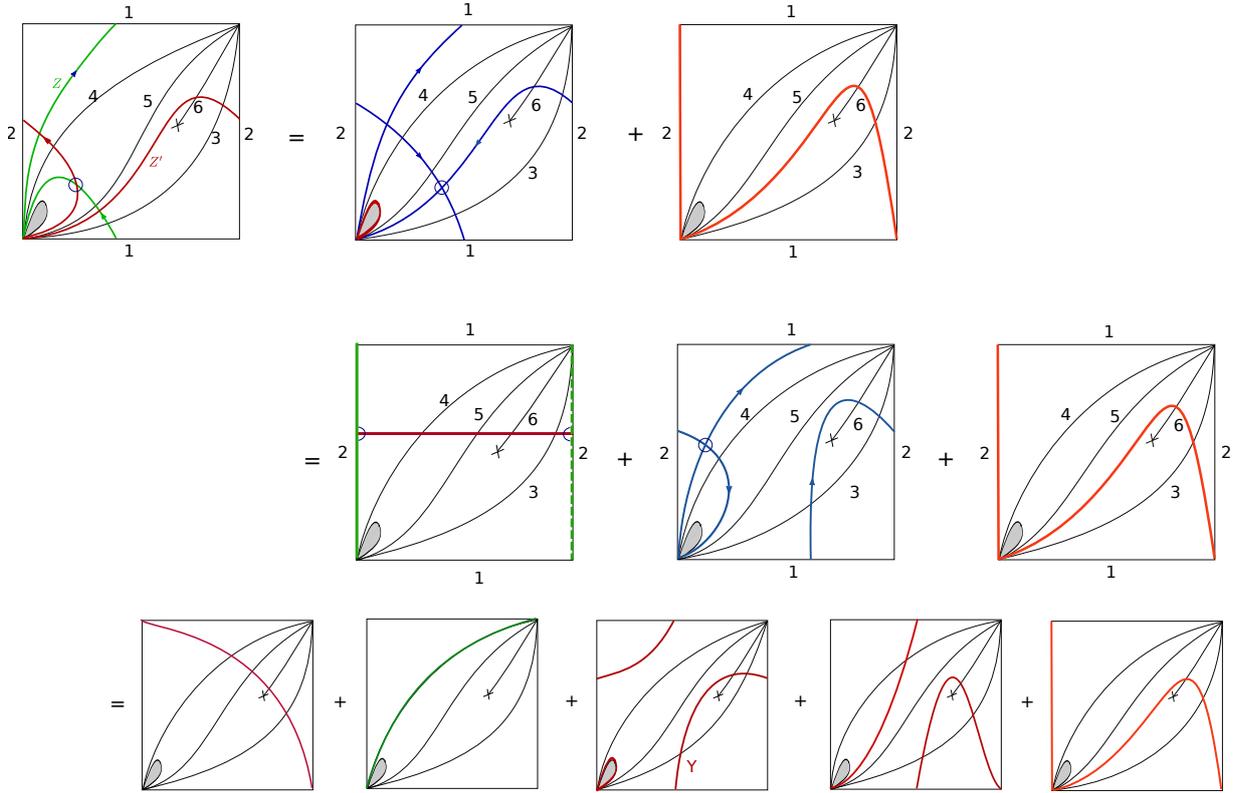}
	\psfrag{Z}{$Z$}
	\psfrag{Z'}{$Z'$}
	\psfrag{Y}{$Y$}
	\caption{Geometric realization of the identities given in Figure~\ref{Fig:y5Y} in terms of arcs in the surface. Step 1: smoothing of the crossing of the arcs $Z$ (green) and $Z'$ (red); step 2: smoothing of a self-crossing; step 3: grafting with the arc $2$ and smoothing of a self-crossing.}
	\label{Fig:SGy5Y}
\end{figure}

\newsavebox{\myboxM}
\sbox{\myboxM}{%
	\begin{tikzpicture}[scale=.5]
	\draw (0,0)--(3,0)--(3,2)--(2,2)--(2,0) (0,0)--(0,1)--(3,1) (1,0)--(1,1);
	\node at (.5,.5){$4$};
	\node at (1.5,.5){$2$};
	\node at (2.5,.5){$3$};
	\node at (2.5,1.5){$6$};
	\node at (0,.5)[left, scale=.6]{$B$};
	\end{tikzpicture}
}

\newsavebox{\myboxN}
\sbox{\myboxN}{%
	\begin{tikzpicture}[scale=.5]
	\draw (0,0)--(1,0)--(1,2)--(0,2)--(0,0) (0,1)--(2,1)--(2,2)--(1,2) (2,1)--(4,1)--(4,2)--(2,2) (3,1)--(3,2)
	(4,1)--(5,1)--(5,3)--(4,3)--(4,2)--(5,2);
	\node at (.5,.5){$1$};
	\node at (.5,1.5){$3$};
	\node at (1.5,1.5){$5$};
	\node at (2.5,1.5){$4$};
	\node at (3.5,1.5){$2$};
	\node at (4.5,1.5){$3$};
	\node at (4.5,2.5){$6$};
	\end{tikzpicture}
}

\newsavebox{\myboxI}
\sbox{\myboxI}{%
	\begin{tikzpicture}[scale=.5]
	\draw (0,0)--(3,0)--(3,1)--(0,1)--(0,0) (1,0)--(1,1) (2,0)--(2,2)--(3,2)--(3,1)--(2,1)--(2,2) (3,1)--(4,1)--(4,2)--(3,2);
	\node at (.5,.5){$2$};
	\node at (1.5,.5){$4$};
	\node at (2.5,.5){$5$};
	\node at (2.5,1.5){$6$};
	\node at (3.5,1.5){$3$};
	\node at (0,0){$\bullet$};
	\node at (0,1){$\bullet$};
	\node at (3,2){$\bullet$};
	\node at (4,2){$\bullet$};
	\end{tikzpicture}
} 

\newsavebox{\myboxQ}
\sbox{\myboxQ}{%
	\begin{tikzpicture}[scale=.5]
	\draw (0,0)--(4,0)--(4,1)--(0,1)--(0,0) (1,0)--(1,1) (2,0)--(2,1) (3,0)--(3,3)--(4,3)--(4,1) (3,2)--(4,2);
	\node at (.5,.5){$1$};
	\node at (1.5,.5){$3$};
	\node at (2.5,.5){$6$};
	\node at (3.5,.5){$3$};
	\node at (3.5,1.5){$2$};
	\node at (3.5,2.5){$4$};
	\end{tikzpicture}
}

\newsavebox{\myboxR}
\sbox{\myboxR}{%
	\begin{tikzpicture}[scale=.5]
	\draw (1,0)--(3,0)--(3,1)--(1,1)--(1,0) (2,0)--(2,2)--(3,2)--(3,1)--(2,1)--(2,2) (3,1)--(4,1)--(4,2)--(3,2);
	\node at (1.5,.5){$4$};
	\node at (2.5,.5){$5$};
	\node at (2.5,1.5){$6$};
	\node at (3.5,1.5){$3$};
	\end{tikzpicture}
}

\newsavebox{\myboxYY}
\sbox{\myboxYY}{%
	\begin{tikzpicture}[scale=.5]
	\draw (1,0)--(2,0)--(2,2)--(1,2)--(1,0) (1,1)--(3,1)--(3,2)--(2,2) (2,2)--(2,3)--(4,3)--(4,2)--(3,2)--(3,3);
	\node at (1.5,.5){$6$};
	\node at (1.5,1.5){$3$};
	\node at (2.5,1.5){$2$};
	\node at (2.5,2.5){$1$};
	\node at (3.5,2.5){$3$}; 
	\node at (1,0){$\bullet$};     
	\node at (2,0){$\bullet$};
	\node at (4,3){$\bullet$};     
	\node at (4,2){$\bullet$};
	\end{tikzpicture}
}

\newsavebox{\myboxP}
\sbox{\myboxP}{%
	\begin{tikzpicture}[scale=.5]
	\draw (0,0)--(5,0)--(5,1)--(0,1)--(0,0) (1,0)--(1,1) (2,0)--(2,1) (3,0)--(3,1) (4,0)--(4,1);
	\node at (.5,.5){$4$};
	\node at (1.5,.5){$1$};
	\node at (2.5,.5){$3$};
	\node at (3.5,.5){$6$};
	\node at (4.5,.5){$3$};
	\end{tikzpicture}
}

\newsavebox{\myboxYBT}
\sbox{\myboxXBT}{%
	\begin{tikzpicture}
	\matrix [matrix of math nodes,left delimiter=(,right delimiter=)]
	{\begin{tikzpicture}
		[scale=.5]
		\node at (0,0){$ \usebox{\myboxYY}$}; 
		\node at (2.3,1.5){$ \usebox{\myboxBB}$};
		\node at (3.3,1.3){$+$};
		\node[red] at (4.3,1.3){$y_2$};
		\node at (7.5,1){$ \usebox{\myboxP}$}; 
		\end{tikzpicture}
		\\
	};
\end{tikzpicture}
}

\begin{figure}[H]
	\scalebox{.82}{\begin{tikzpicture}
		\node at (0,0){$ \usebox{\myboxC}$}; 
		\node at (1.7,.25){$ \usebox{\myboxM}$};
		\node at (3,0){$=$};  
		\node at (4.8,0){$ \usebox{\myboxN}$}; 
		\node at (6.3,0){$+$}; 
		\node[red] at (7.3,0){$y_1y_3y_5$};
		\node at (8.3,.2){$\begin{tikzpicture}[scale=.5]
			\draw (0,0)--(1,0);	
			\node at (.5,0)[above, scale=.6]{$2$};
			\end{tikzpicture}$}; 
		\node at (9,.2){$\begin{tikzpicture}[scale=.5]
			\draw (0,0)--(1,0)--(1,2)--(0,2)--(0,0) (0,1)--(1,1);	
			\node at (.5,.5){$3$};
			\node at (.5,1.5){$6$};
			\end{tikzpicture}$}; 
		
		\node at (3,-2){$=$};  
		\node at (4.5,-2){$\usebox{\myboxI}$}; 
		\node at (5.9,-2){$\begin{tikzpicture}[scale=.5]
			\draw (0,0)--(1,0);	
			\node at (.5,0)[above, scale=.6]{$2$};
			\end{tikzpicture}$}; 
		\node at (6.6,-2){$+$}; 
		\node[red] at (7.2,-2){$y_5$};
		\node at (8.6,-2){$\usebox{\myboxQ}$};
		\node at (10,-2){$+$}; 
		\node[red] at (10.9,-2){$y_1y_3y_5$};
		\node at (11.8,-1.8){$\begin{tikzpicture}[scale=.5]
			\draw (0,0)--(1,0);	
			\node at (.5,0)[above, scale=.6]{$2$};
			\end{tikzpicture}$}; 
		\node at (12.5,-1.8){$\begin{tikzpicture}[scale=.5]
			\draw (0,0)--(1,0)--(1,2)--(0,2)--(0,0) (0,1)--(1,1);	
			\node at (.5,.5){$3$};
			\node at (.5,1.5){$6$};
			\end{tikzpicture}$};

		\node at (3,-5){$=$};
		\node[red] at (3.5,-5){$y_2$};
		\node at (4.6,-5){$\usebox{\myboxR}$};   
		\node at (5.6,-5){$+$};
		\node at (6.3,-5){$\begin{tikzpicture}[scale=.5]
			\draw (0,0)--(1,0);	
			\node at (.5,0)[above, scale=.6]{$4$};
			\end{tikzpicture}$};
		\node at (7,-5){$+$}; 
		\node[red] at (7.5,-5){$y_5$};
		\node at (11.5,-5){$\usebox{\myboxXBT}$};

		\node at (11.2,-7){$+$}; 
		\node[red] at (12,-7){$y_1y_3y_5$};  
		\node at (13,-7){$\begin{tikzpicture}[scale=.5]
			\draw (0,0)--(1,0);	
			\node at (.5,0)[above, scale=.6]{$2$};
			\end{tikzpicture}$}; 
		\node at (13.7,-7){$\begin{tikzpicture}[scale=.5]
			\draw (0,0)--(1,0)--(1,2)--(0,2)--(0,0) (0,1)--(1,1);	
			\node at (.5,.5){$3$};
			\node at (.5,1.5){$6$};
			\end{tikzpicture}$};
		\node at (0,-8){$\quad$};
		\end{tikzpicture}
	}
	\caption{First line: grafting at $B$; second line: self-grafting; third line: grafting with $2$ and self-grafting, respectively.}
	\label{Fig:y5Y}
\end{figure}
\end{proof}
\section{Arbitrary genus case}
\label{Sec:ArbitraryGenus}

In this section, we show the Laurent polynomial associated to the essential loop around the boundary belongs to the cluster algebra with principal coefficients for orbifolds of arbitrary genus and with exactly one marked point on their boundary.

Throughout this section, we work with the triangulation given in Figure~\ref{Fig:HigherGenus} and proceed as in Section~\ref{Sec:Genus1}. The arcs $U$ and $V$ and the resolution of the crossing of these arcs are indicated in Figure~\ref{Fig:HigherGenUV}. 

\begin{figure} [ht]
	\includegraphics[width=.65\textwidth]{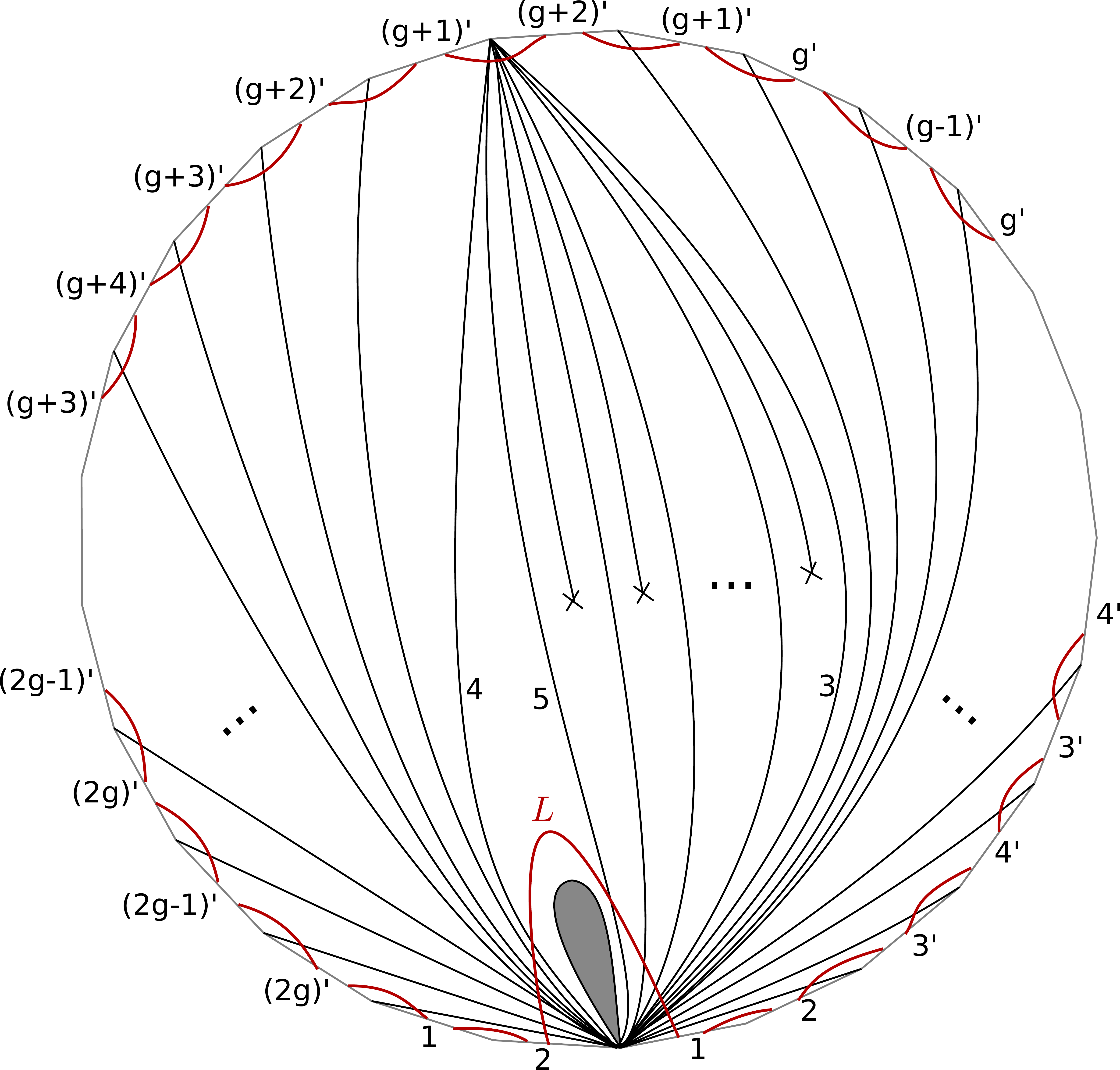}
	\psfrag{L}{$L$}
	\caption{Initial triangulation for higher genus and the essential loop around the boundary $L$. }
	\label{Fig:HigherGenus}
\end{figure}

\begin{figure} [ht]
	\includegraphics[width=.65\textwidth]{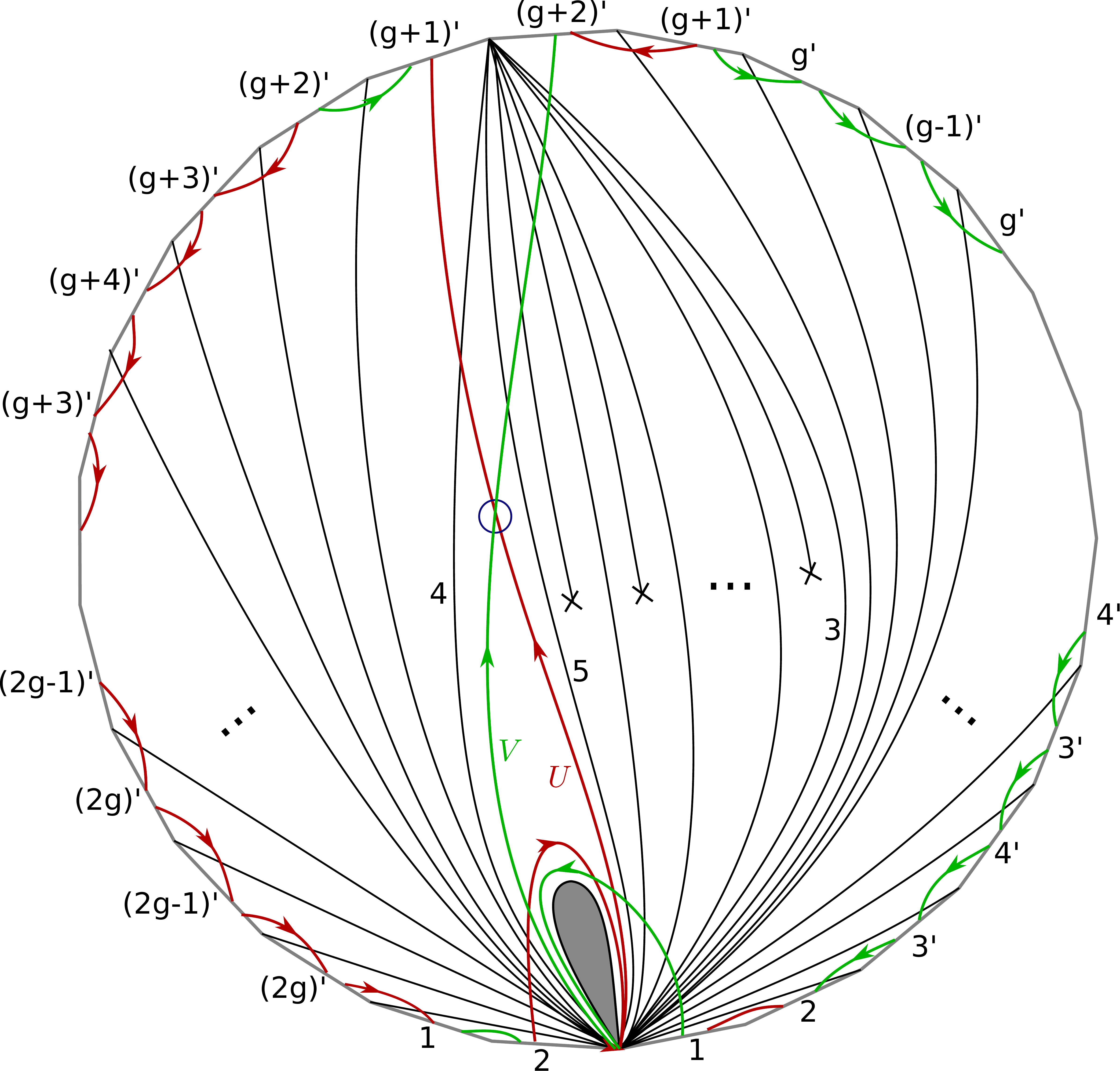}
\caption{The arcs $U$ and $V$ for higher genus.}
\label{Fig:HigherGenUV}
\end{figure}

\begin{figure} [ht]
	\includegraphics[width=1.01\textwidth]{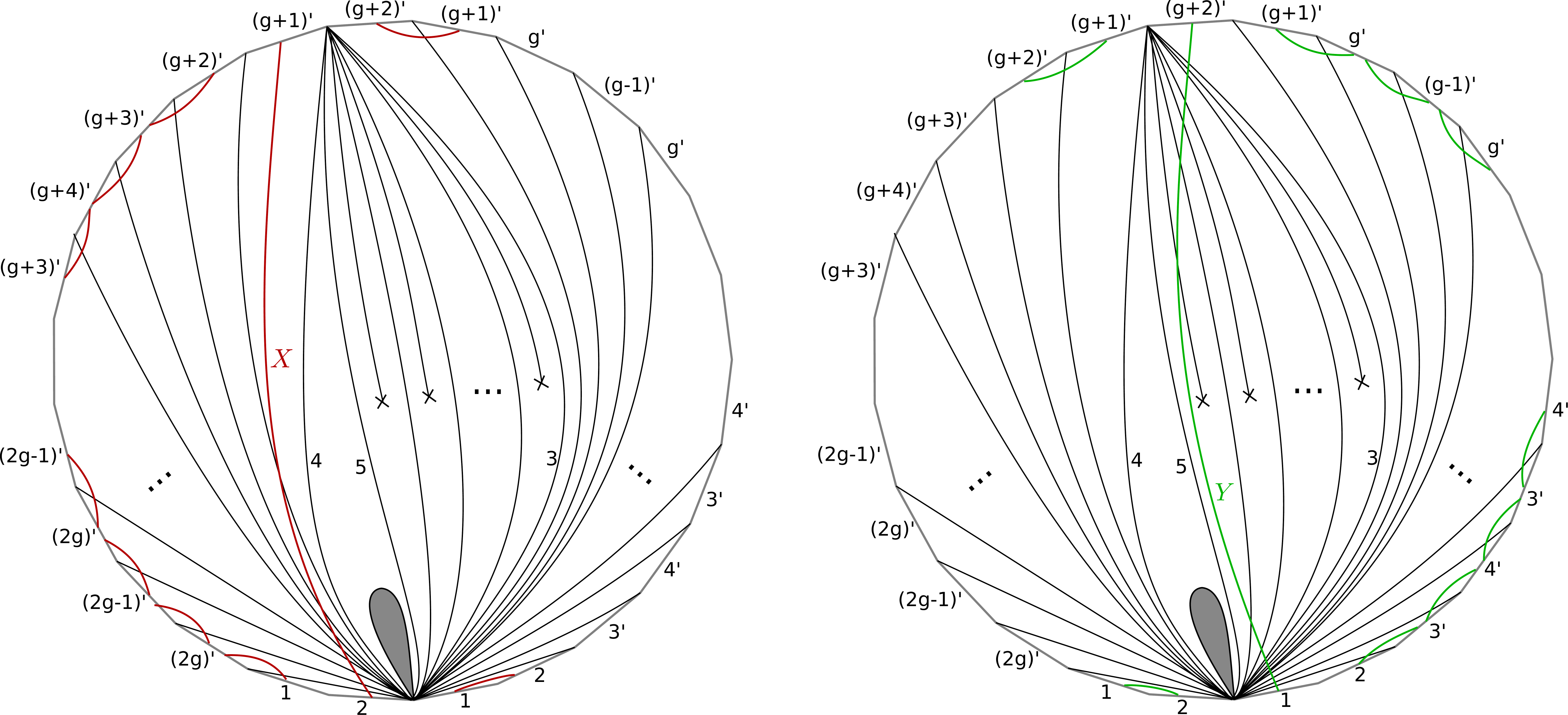}
	\caption{The closed curves $X$ and $Y$ for higher genus.}
	\label{Fig:HigherGenXY}
\end{figure}

\begin{theorem}\label{Thm:LgenArb} Let $\calo$ be an orbifold of arbitrary genus with exactly one boundary marked point, and let $\cala(\calo)$ be its associated cluster algebra with principal coefficients.
	The Laurent polynomial associated to the essential loop around the boundary belongs to the cluster algebra $\cala(\calo).$
\end{theorem}

\begin{proof} The initial triangulation (Figure~\ref{Fig:HigherGenus}) and the arcs $U$ and $V$ (Figure~\ref{Fig:HigherGenUV}) are chosen in a way that the identities in terms of snake graph calculus corresponding to smoothings of the crossings of $U$ and $V$ are locally similar to the genus one case. We also choose labels for the initial triangulation to reflect this occurrence. Indeed, the arcs $U$ and $V$ cross in the triangle bounded by the arcs $4$, $5$ and the boundary arc as in the genus one case. The arcs $X$ and $Y$ obtained in the process of smoothing the crossing between $U$ and $V$ are given in Figure~\ref{Fig:HigherGenXY} (compare with Figure~\ref{Fig:Gen1_T}). Snake graphs associated to the arcs $U, V, X, Y$ and $L$ will have longer zig-zag subgraphs corresponding to the crossings of these arcs with the triangulation of the surface; however, this will again give rise to an identity in the cluster algebra parametrized by the arcs $U, V, X, Y$ and $L$. Namely, 
$$UV=L + (y_4X + x_5) (y_5Y+ y_5\tilde{y}x_4)$$
where $\tilde y= \displaystyle\prod_{i\mbox{ tile in } G_{Q}} y_i$ and the curve $Q$ is obtained by smoothing $U$ and $V$ as in genus $1$ case.
The result then follows by Lemma~\ref{Lem:XYArb} below.
\end{proof}

\begin{lemma} 
\label{Lem:XYArb}
With the notation of Theorem~\ref{Thm:LgenArb}, the following hold.
	\begin{enumerate}
		\item $y_4X\in\cala(\calo),$
		\item $y_5Y\in\cala(\calo).$
	\end{enumerate}
\end{lemma}

\begin{proof} With the choice of initial triangulation given in Figure~\ref{Fig:HigherGenus}, these elements are shown to be in the cluster algebra by following the same steps as in Lemma~\ref{Lem:XY} for genus one case. 

\end{proof}
\section{Proof of the main theorem}
\label{Sec:ArbitraryLoop}

First, we prove the following statement.

\begin{lemma}
\label{Lm-any}
The Laurent polynomial associated to any loop or a semi-closed loop in an orbifold with one marked point belongs to the cluster algebra $\cala(\calo)$. 
\end{lemma}

\begin{proof}
The proof for loops is identical to the one of~\cite[Lemma 4]{CLS}. One can also note using snake graphs calculus that the coefficient of the loop around the boundary in the appropriate relation is equal to one.

For a semi-closed loop (i.e. a geodesic connecting two orbifold points), we can take its small neighborhood and a loop enclosing it, see Figure~\ref{Fig:Semi-closed}. Then the Laurent polynomial of this loop belongs to the cluster algebra as shown above. Now note that this Laurent polynomial is equal to the Laurent polynomial associated to the semi-closed loop (as they have the same geodesic representative, see~\cite[Remark 5.7]{FeTu}). 

\begin{figure} [ht]
	\includegraphics[width=.25\textwidth]{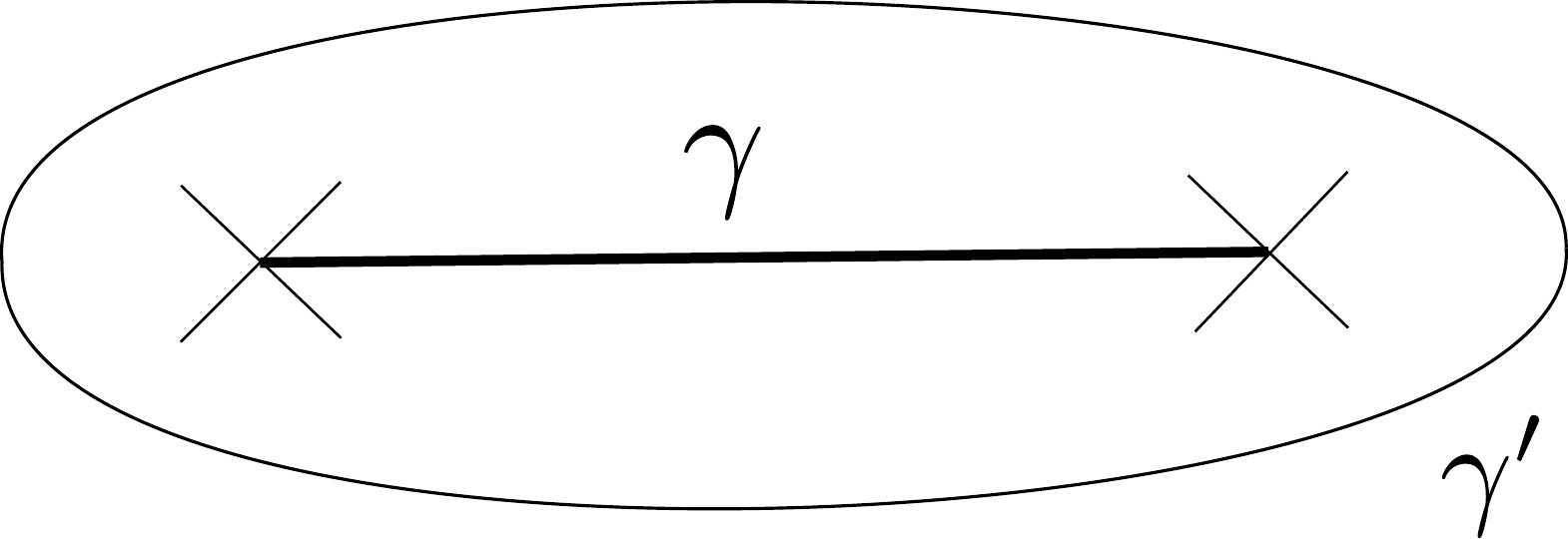}
	\caption{A semi-closed loop $\gamma$ and its small perturbation ${\gamma'}$.}
	\label{Fig:Semi-closed}
\end{figure}
\end{proof}

We can now prove our main result.

\begin{theorem}
\label{Thm:Main}
Let $\cala(\calo)$ be the cluster algebra with principal coefficients associated to an orbifold $\calo$ with exactly one marked point on its boundary. Then the bangles $\calb^\circ$, the bands $\calb^\sigma$, and the bracelets $\calb$ are bases of the cluster algebra $\cala(\calo)$.

\end{theorem}

\begin{proof}
According to~\cite[Section 6.2]{FeTu}, $\calb^\sigma$ and $\calb$ can be obtained from bangles $\calb^\circ$ by a unimodular linear transformation, which means it is sufficient to prove the statement for $\calb^\circ$ only. Elements of $\calb^\circ$ are products of Laurent polynomials associated to mutually compatible arcs, loops and semi-closed loops. According to Lemma~\ref{Lm-any}, all the elements of $\calb^\circ$ belong to $\cala(\calo)$. The linear independence is proved in~\cite[Section 8]{FeTu}, the number of boundary points is irrelevant for the proof. 
\end{proof}

\begin{remark}
One can note that all the properties of the bases $\calb^\circ$, $\calb^\sigma$ and $\calb$ proved in~\cite{FeTu} remain intact for the case of one boundary point (as the proofs did not use the assumption on this number). For example, the bracelets basis $\calb$ is {\em positive}, i.e. it has positive structure constants. 

\end{remark}


\end{document}